\documentclass[12pt]{article}
\usepackage{amsmath,amsthm,amssymb,enumerate,mathtools}
\newtheorem{theorem}{Theorem}[section]

\newtheorem{lemma}[theorem]{Lemma}
\newtheorem{problem}[theorem]{Problem}
\newtheorem{proposition}[theorem]{Proposition}
\newtheorem{corollary}[theorem]{Corollary}
\newtheorem{conjecture}[theorem]{Conjecture}
\theoremstyle{definition}

\newcommand{\bF}{\mathbb F}
\newcommand{\bR}{\mathbb R}

\newcommand{\cB}{\mathcal{B}}

\newcommand{\cF}{\mathcal{F}}

\newcommand{\cL}{\mathcal{L}}

\newcommand{\cH}{\mathcal{H}}
\newcommand{\cM}{\mathcal{M}}
\newcommand{\cN}{\mathcal{N}}

\newcommand{\cS}{\mathcal{S}}

\DeclareMathOperator{\si}{si}

\DeclareMathOperator{\cl}{cl}

\DeclareMathOperator{\PG}{PG}

\newcommand{\del}{\!\setminus\!}
\newcommand{\con}{/}

\newcommand{\ind}{\text{-ind}}
\usepackage[margin=1in]{geometry}

\usepackage[dvipsnames]{xcolor}

\DeclareMathOperator{\ex}{ex}

\newcommand{\Q}{\mathsf{Q}}
\newcommand{\M}{\mathsf{M}}
\newcommand{\sfH}{\mathsf{H}}

\newcommand{\cD}{\mathcal{D}}
\newcommand{\BBG}[3]{\text{BB}(#1,#2,#3)}
\renewcommand{\PG}[2]{\text{PG}(#1,#2)}

\title{Turán densities for matroid basis hypergraphs}

\author{Jorn van der Pol\footnote{University of Twente, Email: j.g.vanderpol@utwente.nl}
\and 
Zach Walsh\footnote{Auburn University, Email: zwalsh@auburn.edu}
\and 
Michael C. Wigal\footnote{University of Illinois at Urbana--Champaign, Email: wigal@illinois.edu}}

\date{\today}

\begin{document}

\maketitle

\begin{abstract}
Let $U$ be a uniform matroid. For all positive integers $n$ and $r$ with $n \ge r$, what is the maximum number of bases of an $n$-element, rank-$r$ matroid without $U$ as a minor? 
We show that this question arises by restricting the problem of determining the Tur\'an number of a daisy hypergraph to the family of matroid basis hypergraphs.
We then answer this question for several interesting choices of $U$.
\end{abstract}

\section{Introduction} \label{sec: introduction}

For a matroid $M$, we write $b(M)$ for the number of bases of $M$. Classic results of Bj\"{o}rner \cite{Bjorner} and Purdy \cite{Purdy} were concerned with minimizing $b(M)$ over all matroids $M$.
This note is concerned with maximizing the number of bases in matroids. Without any restrictions, this problem is trivial, as for any fixed size and rank, a uniform matroid uniquely maximizes the number of bases. To make this problem interesting, we place structural restrictions on the matroids considered. For a matroid $N$, we write 
\[ \ex_{\M}(n,r,N) = \max\{ b(M) : \text{$M$ is an $n$-element, rank-$r$, $N$-minor-free matroid}\}\]
and
\[ \pi_{\M}(r,N) = \lim_{n \to \infty} \frac{\ex_{\M}(n,r,N)}{\binom{n}{r}}.\]
Standard averaging arguments show that $\pi_{\M}(r,N)$ and $\lim_{r \to \infty} \pi_{\M}(r,N)$ exist -- see Propositions~\ref{prop: deletion} and \ref{prop: contraction}. We say that the numbers $\ex_{\M}(n,r,N)$, $\pi_{\M}(r,N)$, and $ \lim_{r \to \infty} \pi_{\M}(r,N)$ are the \emph{Tur{\'a}n basis number}, \emph{Tur{\'a}n basis density}, and \emph{link Tur{\'a}n basis density} of the matroid~$N$. We will focus on the case in which $N$ is a uniform matroid $U_{s,t}$ of rank $s$ on $t$ elements.

\begin{problem}\label{prob: main problem}
Let $s$ and $t$ be positive integers with $t \ge s$.
\begin{enumerate}[$(i)$]
    \item For all integers $n$ and $r$ with $n \ge r \ge s$, determine $\ex_{\M}(n, r, U_{s,t})$.

    \item For all integers $r$ with $r \ge s$, determine $\pi_{\M}(r, U_{s,t})$.

    \item Determine $\displaystyle \lim_{r \to \infty}\pi_{\M}(r, U_{s,t})$.
\end{enumerate}
\end{problem}

While we believe that this problem is interesting from a purely matroidal perspective, it has close connections to several interesting related problems as well.
We comment that while several past works bound the number of bases of a matroid in terms of other parameters such as number of circuits \cite{Ding} or sizes of corcircuits \cite{DouthittOxley}, we bound the number of bases in terms of rank and number of elements.

Our first main result is the following.

\begin{theorem} \label{thm: binary daisies}
Let $n$ and $r$ be integers and let $t$ be a prime power such that  $r \ge 2$. Then
\[ \ex_{\M}(n,r,U_{2,t+2}) \le \frac{n^r \prod_{i=0}^{r-1}(t^r- t^i)}{r! \cdot (t^r - 1)^r}  \]
with equality when $n$ is a multiple of $\frac{t^r - 1}{t-1}$.
In particular,
$$\lim_{r \to\infty}\pi_{\M}(r, U_{2,t+2}) = \prod_{i=1}^{\infty}(1 - t^{-i}).$$
\end{theorem}

Theorem~\ref{thm: binary daisies} was inspired by a recent breakthrough result of Ellis, Ivan, and Leader~\cite{EllisIvanLeader} regarding vertex Tur\'an problems for the hypercube.
They found a construction that gives a lower bound for the limit of Theorem \ref{thm: binary daisies}, which has had wide-reaching consequences throughout combinatorics, see e.g.~\cite{Alon,AxenovichLiu,GrebennikovMarciano}. Our matching upper bound of Theorem~\ref{thm: binary daisies} shows that one must look beyond matroid basis hypergraphs to improve upon their work. We discuss this further in Section~\ref{sec: matroid introducton}.

We also obtain bounds for forbidding $U_{s, s+1}$-minors. 
The number $\ex_{\M}(n, r, U_{s, s+1})$ is particularly natural: it is the maximum number of bases among all $n$-element rank-$r$ matroids with circumference at most $s$, where  \emph{circumference} is the number of elements in a largest circuit of the matroid.
We determine $\pi_{\M}(r, U_{3, 4})$ for all $r \ge 3$.

\begin{theorem} \label{theorem: density for K_4^{(3)}}
Let $r$ be an integer with $r \ge 3$. Then $\pi_{\M}(r, U_{3,4}) = \frac{r!\cdot 2^{\lfloor r/2 \rfloor}}{r^r}$.
\end{theorem}

In particular, $\pi_{\M}(3,U_{3,4}) = 4/9$. We can interpret $\pi_{\M}(s, U_{s,t})$ as a specialization of the well-known hypergraph Tur\'an density parameter $\pi(K_t^{(s)})$, where $K_t^{(s)}$ is the $s$-uniform hypergraph on $t$ vertices. This problem is particularly well-studied in the case $t = s+1$ \cite{BaloghBohmanBollabasZhao, LuZhao, Sidorenko}. We discuss this further in Section~\ref{sec: matroid introducton}.

Finally, in Section \ref{sec: D(3,t)}, we focus on $\pi_{\M}(3, U_{3,t})$. 
Our third main result is as follows.

\begin{theorem} \label{theorem: density for K_5^{(3)}}
$\pi_{\M}(3, U_{3,5}) = 3/4$.
\end{theorem}
This confirms a famous conjecture of Tur\'an \cite{Turan1961}, that $\pi(K_5^{(3)}) = 3/4$, in the special case of matroid basis hypergraphs.
In addition to Theorem \ref{theorem: density for K_5^{(3)}}, we prove two structural results (Theorems \ref{thm: structure for D(3,2m+1)} and \ref{thm: structure for D(3,2m+2)}) for rank-$3$ matroids with no $U_{3,t}$-restriction. These may provide a future approach for finding $\pi_{\M}(3, U_{3,t})$, and may be of independent interest to matroid theorists and finite geometers.

The rest of this note is organized as follows. In Section \ref{sec: matroid introducton}, we will outline preliminaries regarding matroids and hypergraphs. In Section~\ref{sec: matroid hypergraphs} we will show that the limits of Problem \ref{prob: main problem} exist and discuss characterizations of matroid basis hypergraphs, including a characterization via induced Tur\'an numbers as introduced by Loh, Tait, Timmons, and Zhou \cite{LohPoShenTaiTimmonsZhou}.
In Section \ref{sec: U_{1, t} and U_{2,t}} we will prove Theorem \ref{thm: binary daisies}. 
In Section \ref{sec: circumference} we will prove Theorem \ref{theorem: density for K_4^{(3)}} and present a conjecture for $\ex_{\M}(n, U_{s,s+1})$. In Section \ref{sec: D(3,t)} we will prove Theorem \ref{theorem: density for K_5^{(3)}} along with the structural properties of Theorems \ref{thm: structure for D(3,2m+1)} and \ref{thm: structure for D(3,2m+2)}. Finally, in Section \ref{sec: open problems}, we will discuss some related problems and directions for future work.

\section{Preliminaries} \label{sec: matroid introducton}
In this section we will give a brief overview of relevant matroid terminology and how it relates to hypergraphs. We refer the reader to~\cite{Oxley2011} for any notation that remains undefined and for background information regarding matroids.

\subsection{Matroids}
A \emph{matroid} $M$ is a pair $(E, \cB)$ where $E$ is a finite set and $\cB$ is a non-empty family of subsets of $E$ that satisfies the following property:
\begin{enumerate}[(B1)]
\item For all $B_1, B_2 \in \cB$ and $x \in B_1 - B_2$, there exists $y \in B_2 - B_1$ so that $(B_1 - \{x\}) \cup \{y\} \in \cB$.
\end{enumerate}
The set $E$ is the \emph{ground set} of $M$, the sets in $\cB$ are the \emph{bases} of $M$, and property (B1) is the \emph{basis-exchange property}.
We often write $\cB(M)$ for the set of bases of a matroid $M$.
We will take the standard matroid theory convention and write $X - x$ for $X - \{x\}$ for a set $X$ with an element $x$.
It follows from (B1) that all sets in $\cB$ have the same cardinality, which is the \emph{rank} of $M$.
So, every rank-$r$ matroid $M = (E, \cB)$ can be viewed as an $r$-uniform hypergraph with vertex set $E$ and edge set $\cB$; we say that this hypergraph is the \emph{basis hypergraph of $M$}.
We write $\M$ for the class of matroid basis hypergraphs.
Matroid basis hypergraphs are not particularly well-studied, although they do play a central role in several recent works~\cite{BorosGurvichHoMakinoMursic,CameronAjith,Lucchini}.

Let $M = (E, \cB)$ be a matroid.
We say that $M$ is \emph{representable} over a field $\bF$ is there is a matrix over $\bF$ with columns indexed by $E$ so that the elements of $\cB$ are precisely those subsets of columns that form a basis of the column space.
While not every matroid is representable, much of the standard matroid terminology is inspired by this connection to linear algebra.
The \emph{rank} $r(X)$ of a set $X \subseteq E$ is the cardinality of a maximum subset of $X$ contained in a basis.
A \emph{loop} of $M$ is an element that is not in any basis, 
and a \emph{coloop} is an element that is in every basis.
A \emph{circuit} of $M$ is a minimal set of elements that does not appear in any basis.
A \emph{point} of $M$ is a maximal rank-$1$ set, and a \emph{line} of $M$ is a maximal rank-$2$ set.
The \emph{closure} $\cl_M(X)$ of a set $X \subseteq E$ is the maximal set of rank $r(X)$ that contains $X$.
Non-loop elements of $M$ are \emph{parallel} if they are not in a common basis, and the equivalence classes under this relation are the \emph{parallel classes} of $M$.
Restricting the ground set of $M$ to consist of one element from each parallel class gives a new matroid $\si(M)$ called the \emph{simplification} of $M$.
A matroid is \emph{simple} if it is equal to its simplification.
The \emph{deletion} of a non-coloop element $e$ from $M$ is the matroid $M \del e = (E - e, \cB')$, where $\cB' = \{B \in \cB \colon e\notin B\}$, and the \emph{contraction} of a non-loop element $e$ from $M$ is the matroid $M / e = (E - e, \cB'')$, where $\cB'' = \{B - e  \colon e\in B \in \cB\}$.
If $e$ is a coloop, we define $M\del e$ to be $M/e$, and if $e$ is a loop, we define $M/e$ to be $M \del e$.
A matroid $N$ is a \emph{minor} of $M$ if it can be obtained from $M$ by a sequence of deletions and contractions.
Reordering the sequence of deletions and contractions does not change the minor, so for disjoint sets $C, D\subseteq E$ we write $M\con C\del D$ for the matroid obtained from $M$ by contracting the elements in $C$ and deleting the elements in $D$.
We write $M|X$ for $M\del (E - X)$, the \emph{restriction} of $M$ to $X$.
The \emph{dual} of $M = (E, \cB)$ is the matroid $M^*$ on ground set $E$ whose bases are the complements of bases of $M$.
Finally, the \emph{direct sum} of matroids $M_1 = (E_1, \cB_1)$ and $M_2 = (E_2, \cB_2)$ with disjoint ground sets is the matroid $M_1 \oplus M_2 = (E_1 \cup E_2, \cB)$, where $\cB = \{B_1 \cup B_2 \colon B_1 \in \cB_1 \textrm{ and } B_2 \in \cB_2\}$.

\subsection{Hypergraphs}
We will also use some standard notation for hypergraphs.
Given a hypergraph $\cH = (V, F)$, a \emph{subgraph} of $\cH$ is a hypergraph obtained from $\cH$ by deleting vertices and edges, and an \emph{induced subgraph} is a hypergraph obtained from $\cH$ by deleting vertices.
For a hypergraph $\cH = (V, F)$ and $W \subseteq V$ we write $\cH \del W$ for the hypergraph obtained by deleting the vertices in $W$, and we write $\cH/W$ for the \emph{link} of $\cH$ at $W$, which is the hypergraph with vertex set $V - W$ and edge set $\{X - W \colon W \subseteq X \in F\}$.
Note that if $\cH$ is the basis hypergraph of a matroid $M$, then $\cH \del v$ is the basis hypergraph of $M \del v$ (if $v$ is a non-coloop of $M$), and $\cH/v$ is the basis hypergraph of $M/v$ (if $v$ is a non-loop of $M$).
We write $|\cH|$ for the number of edges of $\cH$.
We will refer to $r$-uniform hypergraphs as \emph{$r$-graphs} and to $2$-graphs as graphs.

\subsection{Extremal problems for matroids and hypergraphs}
We now explain how Problem \ref{prob: main problem} arises by restricting Tur\'an parameters to matroid basis hypergraphs.
Given an $r$-graph $\cH$, let $\ex(n,\cH)$ be the maximum number of edges of an $n$-vertex $r$-graph with no subgraph isomorphic to $\cH$.
Then let 
$\pi(\cH) =  \lim_{n\to \infty}\ex(n, \cH)/\binom{n}{r}$; it is well-known that this limit exists.
The parameters $\ex(n, \cH)$ and $\pi(\cH)$ are the \emph{Tur\'an number} and \emph{Tur\'an density}, respectively, of $\cH$.

When $r = 2$, Tur\'an's classical theorem~\cite{Turan} shows that $\ex(n, K_t) \le (1 - \frac{1}{t-1})\binom{n}{2}$ for all integers $n$ and $t$ with $n, t\ge 3$, with equality for balanced complete $(t-1)$-partite graphs, which implies that $\pi(K_t) = 1 - \frac{1}{t-1}$ for all $t \ge 3$.
When $r \ge 3$, determining Tur\'an numbers and densities for cliques becomes much more difficult.
Amazingly, $\pi(K^{(s)}_t)$ is not known for any $t > s \ge 3$, where $K^{(s)}_t$ is the complete $s$-uniform graph on $t$ vertices.
We direct the reader to Keevash's comprehensive survey \cite{Keevash} and the more recent survey of Balogh, Clemen, and Lidick\'y \cite{BaloghClemenLidicky} for details on past results for hypergraph Tur\'an densities.

In 2011, Bollab\'as, Leader, and Malvenuto  \cite{BollabasLeaderMalvenuto} defined a natural generalization of $K_t^{(s)}\!\!\!\!$.
For all integers $r,s,t$ with $r,t  \ge s \ge 1$, an \emph{$(r,s,t)$-daisy}, denoted $\cD_r(s,t)$, is an $r$-uniform hypergraph with vertex set $S \cup T$ and edge set $\{S \cup X \colon X \in T^{(s)}\}$, where $S$ is an $(r - s)$-element set, $T$ is a $t$-element set disjoint from $S$, and $T^{(s)}$ is the set of all $s$-element subsets of $T$.
The set $S$ is the \emph{stem} and the sets in $T^{(s)}$ are the \emph{petals}.
Note that $\cD_s(s, t)$ is simply $K^{(s)}_t$, and as $r$ grows with $s$ and $t$ fixed, $\cD_r(s,t)$ becomes sparser because the number of edges remains constant.
Motivated by vertex Tur\'an problems on the hypercube, Bollob\'as, Leader, and Malvenuto conjectured that $\lim_{r \to \infty} \pi(\cD_r(s,t)) = 0$ for all integers $s$ and $t$ with $t \ge s \ge 1$. 
(It is straightforward to show that $\lim_{r\to \infty}\pi(\cD_r(s,t))$ exists; see Propositions~\ref{prop: deletion} and~\ref{prop: contraction}.)  
Ellis, Ivan, and Leader \cite{EllisIvanLeader} recently disproved this conjecture by showing that $\lim_{r \to \infty} \cD_r(2,4) \ge \prod_{i=1}^{\infty}(1 - 2^{-i}) \approx 0.29$, which implies that $\lim_{r \to \infty} \cD_r(s,t) \gtrsim 0.29$ whenever $s \ge 2$ and $t \ge s + 2$. In fact they proved something more general.

\begin{theorem}[Ellis, Ivan, and Leader \cite{EllisIvanLeader}] \label{thm: ellis-ivan-leader prime power lower bound}
If $t$ is a prime power, then 
$$\lim_{r \to \infty}\pi(\cD_r(2, t+2)) \ge \prod_{i=1}^{\infty}(1 - t^{-i}).$$
\end{theorem}

For all integers $r$ and $n$ with $r \ge 2$ and $n = k\cdot \frac{t^r - 1}{t-1}$ for a positive integer $k$, they consider the $n$-vertex $r$-uniform hypergraph whose vertices are points in $r$-dimensional projective space over the finite field $\bF_t$ with $k$ copies of each vector, and whose edges are bases, and they show that this hypergraph has no $\cD_r(2,t+2)$-subgraph.

We now present an attractive matroidal reinterpretation of the construction by Ellis, Ivan, and Leader. First, we relate matroid minors to hypergraph suspensions (also known as hypergraph augmentations), see e.g.~\cite{Alon,BaloghHalfpapLidickyPalmer,Mukherjee}.
For any $s$-uniform hypergraph $\cH$ and integer $r \ge s$, the \emph{rank-$r$ suspension} of $\cH$, denoted $\cS_r(\cH)$, is the $r$-uniform hypergraph with vertex set $S \cup V(\cH)$ with $|S| = r -s$ and edge set  $\{S \cup X \colon X \in E(\cH)\}$.  The set $S$ is the \emph{stem} of $\cS_r(\cH)$. We see that $\cD_r(s, t) = \cS_r(K_t^{(s)})$.

For an $n$-element rank-$s$ matroid $N$ with basis hypergraph $\cH$ and an integer $r \ge s$, define $\mathsf{F}_r(N)$ to be the set of all $r$-uniform hypergraphs $\mathcal{K}$ on $n+r-s$ vertices that contain $\cS_r(\cH)$ with stem $S$ as a subgraph, with the property that whenever an edge $e$ of $\mathcal{K}$ contains $S$, then $e$ is an edge of $\cS_r(\cH)$ (this is equivalent to the link of $\mathcal{K}$ at $S$ being a copy of $\mathcal{H}$).

\begin{proposition}\label{prop: suspension_free_minor_free}
Let $r \ge s \ge 1$ be integers and let $N$ be a rank-$s$ matroid. For every rank-$r$ matroid $M$ with basis hypergraph $\cF$, $M$ has an $N$-minor if and only if $\cF$ has an induced subgraph belonging to $\mathsf{F}_r(N)$. If, in addition, $N$ is a uniform matroid, then $\cF$ contains $\cS_r(\cH)$ as a subgraph if and only if $M$ has an $N$-minor.
\end{proposition}
\begin{proof}
    First, suppose that $M$ has an $N$-minor. Then there are disjoint sets $E_1, E_2 \subseteq E(M)$ so that $M\del E_1/E_2 = N$ and $E_2$ is independent in $M$ \cite[Lemma 3.3.2]{Oxley2011}.
    Since $E_2$ is independent, the bases of $N$ are precisely the subsets of $E(N)$ whose union with $E_2$ is a basis of $M$.
    Equivalently, the link of $\cF\del E_1$ at $E_2$ is a copy of $\cH$.
    Therefore, the restriction of $\cF$ to the vertex set $E(M) \del E_1$ contains a copy of $\cS_r(\cH)$ with stem $E_2$. Therefore the induced subgraph of $\cF$ corresponding to $E(M) \del E_1$ belongs to $\mathsf{F}_r(N)$.

    Now suppose $\cF$ has an induced subgraph $\cD$ belonging to $\mathsf{F}_r(N)$. Then $\cD$ contains $\cS_r(\cH)$ as a subgraph with a stem $S$. Deleting the vertices outside of $\cD$ and taking the link at $S$ gives $\cH$. These operations correspond to deleting elements of $M$ outside of $\cD$ and contracting the elements in $S$ to obtain a copy of $N$. That is, $M$ has an $N$-minor. 

    When $N$ is uniform, it suffices to argue that if $\cF$ has $\cS_r(\cH)$ as a subgraph, then $M$ has $N$ as a minor. Indeed, if $S$ is the stem of $\cS_r(\cH)$, deleting the vertices outside $\cS_r(\cH)$ and taking the link at $S$ yields $\cH$. 
\end{proof}

An application of Proposition~\ref{prop: suspension_free_minor_free} to uniform matroids shows that a rank-$r$ matroid is $U_{s,t}$-minor-free if and only its basis hypergraph does not have the daisy $\cD_r(s,t) = \cS_r(K_t^{(s)})$ as a subgraph.

\begin{corollary} \label{cor: uniform matroid minors and daisy subgraphs}
Let $r, s, t$ be integers with $r,t \ge s \ge 1$, and let $\cH$ be the basis hypergraph of a matroid $M$.
Then $\cH$ has $\cD_r(s,t)$ as a subgraph if and only if $M$ has a $U_{s,t}$-minor.
\end{corollary}

Therefore Problem~\ref{prob: main problem} arises by restricting hypergraph Tur\'an parameters to matroid basis hypergraphs.
From this perspective, $\cD_r(2,t+2)$-subgraph-freeness of the hypergraph constructed by Ellis, Ivan, and Leader is immediate, as the associated matroid is $\bF_t$-representable and therefore has no $U_{2,t+2}$-minor (see e.g.\ \cite[Corollary~6.5.3]{Oxley2011}).
Moreover, Ellis, Ivan, and Leader show that the same hypergraph is $\cD_r(t, t+2)$-subgraph-free in \cite[Proposition 1]{EllisIvanLeader_smallhypercubes}.
From the matroid perspective, this follows immediately from the fact that $U_{t,t+2}$ is the dual matroid of $U_{2, t+2}$, and the class of $\bF_t$-representable matroids is closed under duality.

\section{Matroid basis hypergraphs} \label{sec: matroid hypergraphs}

In this section we will show that the limits of Problem \ref{prob: main problem} exist, and describe -- in terms of forbidden substructures -- how the class of rank-$r$ matroid basis hypergraph sits within the class of all $r$-uniform hypergraphs.

\subsection{Density}

Given families $\Q$ and $\sfH$ of hypergraphs, let $\ex_{\Q}(n, r, \sfH)$ be the maximum number of edges of an $r$-uniform, $n$-vertex, hypergraph in $\Q$ with no subgraph in $\sfH$. 
We set $\ex_{\Q}(n, r, \sfH) = 0$ if no such hypergraph exists.
Furthermore, we allow $\sfH$ to be empty; in this case, $\ex_{\Q}(n, r, \varnothing)$ simply denotes the maximum number of edges of an $r$-uniform, $n$-vertex, hypergraph in $\Q$.
Let $\pi_{\Q}(r, \sfH) = \lim_{n\to \infty}\ex_{\mathsf Q}(n, r, \sfH)/\binom{n}{r}$, if this limit exists.
Note that when $\Q$ is the class of $r$-uniform hypergraphs and $\sfH = \{\cH\}$ for an $r$-graph $\cH$, then $\ex_{\Q}(n, r, \sfH)$ is equal to the classical Tur\'an number $\ex(n, \cH)$.
We will show that $\pi_{\Q}(r, \sfH)$ exists whenever $\Q$ is closed under vertex deletion.
The following argument is standard, but we include it for completeness.

\begin{proposition} \label{prop: deletion}
Let $\Q$ and $\sfH$ be families of hypergraphs, and let $r$ be an integer with $r \ge 1$.
If $\Q$ is closed under vertex deletion, then $\pi_{\Q}(r, \sfH)$ exists.
\end{proposition}
\begin{proof}
We will show that the sequence $\{{\ex_{\Q}(n, r, \sfH)}/{\binom{n}{r}} \colon n \ge r\}$ is monotone decreasing.
Let $n \ge r+1$.
If $\ex_{\Q}(n, r, \sfH) = 0$, then $\ex_{\Q}(m, r, \sfH) = 0$ for all $m \ge n$ because $\Q$ is closed under vertex deletion.
So let $\cF \in \Q$ be an $n$-vertex $r$-graph with no subgraph in $\sfH$ and with $|\cF| = \ex_{\Q}(n, r, \sfH)$.
By computing the average edge density of all single-vertex deletions of $\cF$, we see that 
\begin{align*}
\frac{1}{n} \sum_{v \in V(\cF)}\frac{|\cF\del v|}{\binom{n-1}{r}}
=\frac{1}{n} \cdot \frac{(n-r)|\cF|}{\binom{n-1}{r}} 
=\frac{|\cF|}{\binom{n}{r}} = \frac{\ex_{\Q}(n, r, \sfH)}{\binom{n}{r}}.
\end{align*}
Therefore there is some vertex $v$ of $\cF$ so that $|\cF\del v|/\binom{n-1}{r} \ge \ex_{\Q}(n, r, \sfH)/\binom{n}{r}$. As $\Q$ is closed under vertex deletion,  $\cF\del v \in \Q$, thus $\ex_{\Q}(n-1, r, \sfH)/\binom{n-1}{r} \ge \ex_{\Q}(n, r, \sfH)/\binom{n}{r}$, as desired.
\end{proof}

The next result uses the notion of hypergraph suspension as defined in Section \ref{sec: matroid introducton}. 
Given a family $\sfH$ of $s$-graphs and an integer $r \ge s$, let $\cS_r(\sfH) = \{\cS_r(\cH) \colon \cH \in \sfH\}$.

\begin{proposition} \label{prop: contraction}
Let $\Q$ be a family of hypergraphs and let $\sfH$ be a family of $s$-graphs.
If $\Q$ is closed under vertex deletion and taking links, then $\lim_{r\to\infty}\pi_{\Q}(r, \cS_r(\sfH))$ exists.
\end{proposition}
\begin{proof}
We will show that the sequence $\{\pi_{\Q}(r, \cS_r(\sfH)) \colon r \ge s\}$ is monotone decreasing.
Fix some $r \ge s+1$ and $n \ge r+1$.
If $\pi_{\Q}(r, \cS_r(\sfH)) = 0$, then $\pi_{\Q}(t, \cS_t(\sfH)) = 0$ for all $t \ge r$ because $\Q$ is closed under taking links.
So let $\cF \in \Q$ be an $n$-vertex $r$-graph with no subgraph in $\cS_r(\sfH)$ and with $|\cF| = \ex_{\Q}(n, \cS_r(\sfH))$.
By computing the average basis density of all link graphs of $\cF$, we see that 
\begin{align*}
\frac{1}{n} \sum_{v \in  V(\cF)}\frac{|\cF/ v|}{\binom{n-1}{r-1}}
=\frac{1}{n} \cdot \frac{r|\cF|}{\binom{n-1}{r-1}} 
=\frac{|\cF|}{\binom{n}{r}} = \frac{\ex_{\Q}(n, \cS_r(\cH))}{\binom{n}{r}}.
\end{align*}
 Therefore there is some vertex $v$ of $\cF$ so that $|\cF/v|/\binom{n-1}{r-1} \ge \ex_{\Q}(n, \cS_r(\sfH))/\binom{n}{r}$.
In particular, $v$ is in at least one edge, so $\cF/v$ is a nonempty $(r-1)$-graph.
Since $\cF/v \in \Q$ because $\Q$ is closed under taking links, it follows that $\ex_{\Q}(n-1, r, \cS_{r-1}(\sfH))/\binom{n-1}{r} \ge \ex_{\Q}(n, r, \cS_r(\sfH))/\binom{n}{r}$.
This implies that the sequence $\{\pi_{\Q}(r, \cS_r(\sfH)) \colon r \ge s\}$ is monotone decreasing, as desired.
\end{proof}

We have the following corollary for matroid basis hypergraphs.
\begin{corollary} \label{cor: matroid basis densities exist}
If $N$ is a matroid, then $\pi_{\M}(r, N)$ exists for all integers $r \ge 1$, and $\lim_{r \to\infty}\pi_{\M}(r, N)$ exists.
\end{corollary}
\begin{proof}
Let $\cM$ be the class of $N$-minor-free matroids, and let $\Q$ be the class of basis hypergraphs of matroids in $\cM$.
Then $\pi_{\M}(r, N) = \pi_{\Q}(r, \varnothing)$.
Note that $\Q$ is closed under vertex deletion and taking links because $\cM$ is closed under deletion and contraction.
Therefore $\pi_{\M}(r, N)$ exists for all $r \ge 1$ by Proposition~\ref{prop: deletion}, and $\lim_{r \to\infty}\pi_{\M}(r, N)$ exists by Proposition~\ref{prop: contraction}.
\end{proof}

Note that the same proof shows that if $\cM$ is any minor-closed class of matroids and $\Q$ is the class of basis hypergraphs of matroids in $\cM$, then $\pi_{\Q}(r, \varnothing)$ and $\lim_{r\to\infty}\pi_{\M}(r, \varnothing)$ both exist.
In this case we write $\pi_{\cM}(r)$ for $\pi_{\Q}(r, \varnothing)$.

\subsection{Forbidden substructures}\label{sec: forbidden substructures}
Since we will be specializing problems for $r$-uniform hypergraphs to the class $\M$ of matroid basis hypergraphs, we will briefly discuss how $\M$ fits into the class of all uniform hypergraphs.
Since $\M$ is closed under vertex deletion and taking links, it is natural to ask for a characterization of $\M$ via a list of hypergraphs that are not in $\M$ but have every vertex deletion and every link in $\M$.
This list was found by Cordovil, Fukuda, and Moreira \cite{CordovilFukudaMoreira}; it consists of three $4$-vertex graphs, and the infinite family of $2$-edge $r$-graphs whose vertex set is the disjoint union of the two edges.

It is also natural to ask for a characterization of the class $\M$ by forbidden induced subgraphs.
It is straightforward to show that for each $r \ge 2$, the class of $r$-uniform matroid basis hypergraphs can be characterized by a finite collection of forbidden induced subgraphs; we omit the proof.

\begin{proposition} \label{prop: matroidal forbidden induced subgraph}
An $r$-uniform hypergraph is a matroid basis hypergraph if and only if every induced subgraph with at most $2r$ vertices is a matroid basis hypergraph.
\end{proposition}

Let $\mathsf{F}$ and $\mathsf{H}$ denote families of $r$-uniform hypergraphs, and let $\ex(n,\{\mathsf{H},\mathsf{F}\ind\})$ denote the maximum number of edges in an $n$-vertex $r$-uniform hypergraph containing no subgraph isomorphic to a hypergraph belonging to $\mathsf{H}$ and no induced subgraph isomorphic to a hypergraph belonging to $\mathsf{F}$.  If $N$ is a rank-$s$ uniform matroid on $t$ elements, then its basis hypergraph is $K_t^{(s)}$. An immediate consequence of Corollary~\ref{cor: uniform matroid minors and daisy subgraphs} and Proposition~\ref{prop: matroidal forbidden induced subgraph} is for $r \ge s$, 
\[ \ex_{\M}(n,r,N) = \ex(n,\{\cD_r(s, t) ,\mathsf{F}\ind\}),\]
for an appropriate family $\mathsf{F}$ of $r$-uniform hypergraphs. In the case when $H$ and $F$ are graphs, the problem of determining $\ex(n,\{H,F\ind\})$ was introduced by Loh, Tait, Timmons, and Zhou \cite{LohPoShenTaiTimmonsZhou} in 2018  and has received significant attention since.

While the number of forbidden induced subgraphs for $r$-uniform matroid basis hypergraphs is finite for all $r$, it grows exponentially as a function of $r$.
Write $B_n$ for the $n$-th Bell number, which counts the number of partitions of an $n$-element set, and write $m(n,2)$ for the number of rank-$2$ matroids on an $n$-element set. While better asymptotics for $B_n$ are known, we will use only the trivial bound $B_n \le n^n$. Every rank-2 matroid on $[n]$ can be described as a set partition of $[n+1]$ (in which the class containing $n+1$ denotes the set of loops), so $m(n,2) \le B_{n+1} \le (n+1)^{n+1}$.

\begin{proposition} \label{prop: exponentially many induced subgraphs}
The class of $r$-uniform matroid basis hypergraphs has at least $2^{r^2/2 - O(r\log(r))}$ non-isomorphic forbidden induced subgraphs on $r+2$ vertices.
\end{proposition}
\begin{proof}
We first show that every non-empty $r$-graph $\cH$ with at most $r+1$ vertices is in $\M$.
If $\cH$ has at most $r$ vertices, then it has at most one edge, so its edges trivially satisfy the basis exchange axiom.
So we may assume that $\cH$ has $r+1$ vertices.
Let $F_1$ and $F_2$ be edges of $\cH$.
Note that $|F_1 - F_2| = |F_2 - F_1| = 1$.
Let $v_1 \in F_1 - F_2$ and let $v_2 \in F_2 - F_1$.
Since $(F_1 - v_1) \cup v_2$ and $(F_2 - v_2) \cup v_1$ are both edges (the former is $F_2$ and the latter is $F_1$), the edges of $\cH$ satisfy the basis exchange axiom, so $\cH \in \M$.

Consequently, an $r$-graph $\cH$ on $r+2$ vertices is a forbidden induced subgraph if and only if it is not in $\M$.
The number of $r$-graphs in $\M$ on $r+2$ vertices is equal to the number of $(r+2)$-element matroids of rank~$r$, which, by matroid duality, is equal to the number of $(r+2)$-element matroids of rank~$2$. It follows that the number $f_r$ of non-isomorphic forbidden induced subgraphs on $r+2$ vertices satisfies
\begin{equation*}
	f_r \ge \frac{1}{(r+2)!}\left(2^{\binom{r+2}{r}}-m(r+2,2)\right) \ge \frac{1}{(r+2)!}\left(2^{\binom{r+2}{r}}-(r+3)^{r+3}\right).
\end{equation*}
The proposition then follows from Stirling's approximation.
\end{proof}

The proof of Proposition~\ref{prop: exponentially many induced subgraphs} shows that, as $r\to\infty$, almost every $r$-graph on $r+2$ vertices is a forbidden induced subgraph for the class of matroidal hypergraphs. The actual number of forbidden induced subgraphs is likely much larger than this lower bound, because the proof only considers $r$-graphs on $r+2$ vertices, while forbidden induced subgraphs can have up to $2r$ vertices.

Finally, we remark that a typical $r$-graph is not the basis hypergraph of a matroid, in the sense that
\begin{equation*}
    \lim_{n \to \infty} \frac{m(n,r)}{2^{\binom{n}{r}}} = 0
    \qquad
    \text{for all $r \ge 2$},
\end{equation*}
where $m(n,r)$ is the number of $r$-uniform matroid basis hypergraphs on a given set of $n$ vertices. For $r = 2$, this follows immediately from the above observation that $m(n,2) \le (n+1)^{n+1}$, while for $r \ge 3$ this follows from the bound $\log m(n,r) = O\left(\frac{\log n}{n}\binom{n}{r}\right)$ as $n\to\infty$, see e.g.~\cite{BansalPendavinghVanderPol2014}.

\section{$U_{1, t}$ and $U_{2, t}$} \label{sec: U_{1, t} and U_{2,t}}

In this section we consider $\pi_{\M}(r,U_{s, t})$ with $s \in \{1,2\}$, and prove Theorem \ref{thm: binary daisies}. 
We begin with $\pi_{\M}(r,U_{1,t})$, which is straightforward to calculate.

\begin{proposition} \label{prop: s = 1}
Let $r \ge 1$ and $t \ge 2$ be integers. Then
\[\ex_{\M}(n,r,U_{1,t}) \le (t-1)^r, \]
with equality when $t-1$ divides $n$. 
\end{proposition}
\begin{proof}
We proceed by induction on $r$.
A matroid of rank $1$ with no $U_{1,t}$-restriction has at most $t-1$ non-loop elements, and therefore at most $t-1$ bases. Thus the bound holds when $r = 1$, establishing the base case of the induction.
Now let $r > 1$ and assume that the claim holds for all rank-$(r-1)$ matroids. Let $M$ be a rank-$r$ matroid with no $U_{1,t}$-minor and let $C^*$ be a largest cocircuit (a circuit of the dual) of $M$. As $M$ has no $U_{1,t}$-minor, we have $|C^*| \le t-1$. Furthermore, as $C^*$ is a cocircuit, every basis of $M$ contains an element of $C^*$ (see \cite[Proposition 2.1.19]{Oxley2011}) and $C^*$ contains no loops of $M$.
For each element $e \in C^*$, as $e$ is not a loop of $M$, the matroid $M/e$ has rank $r-1$ and therefore at most $(t-1)^{r-1}$ bases, by induction.
Since $|C^*| \le t-1$, it follows that $M$ has at most $(t-1)(t-1)^{r-1} = (t-1)^r$ bases, as desired.
Finally, we show that the bound is sharp when $t - 1$ divides $n$.
If $t - 1$ divides $n$, let $M$ be the union of $r$ parallel classes with cardinality $t - 1$ so that the bases of $M$ are the transversals of these classes, i.e., $M$ is a parallel extension of an $r$-element independent set. 
Then $M$ has $n$ elements, rank $r$, and exactly $(t-1)^r$ bases.
\end{proof}

Proposition~\ref{prop: s = 1} implies that $\pi_{\M}(r, U_{1,t}) = 0$ for all $t  \ge 2$ and $r \ge 1$. We comment that this already follows from the stronger result that $\ex(n, \cD_r(1,t)) \le \frac{t-1}{n-r+1}\binom{n}{r}$. To see why the latter result holds, note that if $\cH$ is an $r$-uniform hypergraph with $n$ vertices that does not contain $\cD_r(1,t)$ as a subgraph, then every $r-1$ vertices of $\cH$ are contained in at most $t-1$ edges; it follows that $\cH$ has at most $\binom{n}{r-1}\frac{t-1}{r}$ edges, from which the result follows.

We next turn our attention to $\pi_{\M}(r,U_{2,t})$. 
We will rely on the following theorem of Kung~\cite{Kung1993}.

\begin{theorem}[Kung \cite{Kung1993}]\label{thm: Kung}
For all integers $r$ and $t$ with $r,t\ge 2$, every simple rank-$r$ matroid with no $U_{2,t+2}$-minor has at most $\frac{t^r - 1}{t - 1}$ elements.
\end{theorem}

For general matroids, Theorem~\ref{thm: Kung} implies that any matroid forbidding $U_{2,t+2}$ as a minor has a bounded number of parallel classes.
In general, $U_{s,t}$-minor-free matroids need not have this property.  
For example, one can take a direct sum of $r/2$ copies of $U_{2, \ell}$ to obtain a simple rank-$r$ matroid with $r \cdot \ell$ points and no $U_{s,t}$-minor with $s \ge 3$ and $t \ge 5$.

For $r \ge 1$ and $t\ge 2$, define
\begin{equation}\label{eq:b_definition}
    b(r,t) = \frac{\prod_{i=0}^{r-1}(t^r- t^i)}{r! \cdot (t-1)^r} = \frac{\prod_{i=0}^{r-1}(\frac{t^r - 1}{t-1} - \frac{t^i - 1}{t-1})}{r!}.
\end{equation}

When $t$ is a prime power, $b(r,t)$ is the number of bases of the rank-$r$ projective geometry $\PG{r-1}{t}$ (see e.g.~\cite[Proposition 6.1.4]{Oxley2011}). 
Noting that $b(1,t) = 1$ for all $t$, the following recursion holds for all $r \ge 2$ and $t \ge 2$: 
\begin{equation}\label{eq:b_recursion}
    b(r,t) = \frac{b(r-1,t)\cdot t^{r-1}(t^r - 1)}{r\cdot (t-1)}.
\end{equation}

Our calculation of $\pi_{\M}(r, U_{2, t+2})$ will rely on the theory of hypergraph Lagrangians.
The \emph{Lagrangian polynomial} of a hypergraph $\mathcal{H}$ with vertex set $[n]$ is the polynomial 
\[ p_{\cH}(x_1,\ldots,x_n) = \sum_{e \in E(\mathcal{H})} \prod_{i \in e} x_i.\]
The \emph{Lagrangian} of a hypergraph $\mathcal{H}$ with vertex set $[n]$, denoted $\lambda(\cH)$, is the maximum value of $p_{\cH}(x_1,\ldots,x_n)$ over the standard simplex $\{x_i \ge 0 : \sum_i x_i = 1\}$. By continuity, as the standard simplex is bounded and closed, $\lambda(\cH)$ is well-defined.
For a matroid $M$ with basis hypergraph $\cH$, we write $p_{M}$ for $p_{\cH}$ and $\lambda(M)$ for $\lambda(\cH)$. 

An immediate observation is that $\lambda(M) = \lambda(\si(M))$ for any matroid $M$ on ground set $[s]$. Clearly deleting a loop does not change $\lambda(M)$, because loops are not in any bases, and if elements $1$ and $2$ of $M$ are parallel, then $p_M(x_1, x_2, x_3, \dots, x_s) = p_{M \del 1}(x_1 + x_2, x_3, \dots, x_s)$ because $1 \cup J \in \cB(M)$ if and only if $2 \cup J \in \cB(M)$ for any $J \subseteq [s]$. 
The following lemma relates derivatives of $p_M$ to matroid contraction.

\begin{lemma}\label{lem: Lagrangian gradient}
    Let $r \ge 2$, let $M$ be a rank-$r$ matroid on a ground set $[s]$ with  Lagrangian polynomial $p_M$, and let $i \in [s]$. Then
     \[\frac{\partial}{\partial x_i} p_M\left(x_1,\ldots, x_s\right) 
     \le (1 - x_i)^{r-1} \cdot \lambda(M/i)\]
     for all nonnegative reals $x_1,\ldots,x_s$ such that $\sum_j x_j = 1$.
\end{lemma}
    \begin{proof}
    If $i$ is a loop of $M$, then $p_M$ does not depend on $x_i$, so the left-hand side is equal to~0 and the claim holds trivially.
    So we may assume that $i$ is not a loop of $M$.
    In that case, the bases of $M$ that contain $i$ are in one-to-one correspondence with the bases of $M/i$.
    We see that
    \begin{align*}
        \frac{\partial}{\partial x_i} p_M\left(x_1,\ldots, x_s\right) &=
        \frac{\partial}{\partial x_i}\left( \sum_{B \in \cB(M)} \prod_{j \in B}x_j \right) = \sum_{B \in \cB(M)\colon i \in B} \prod_{j \in B - i}x_j \\
        &= \sum_{B' \in \cB(M/i)} \prod_{j \in B'}x_j = (1 - x_i)^{r-1} \cdot \sum_{B' \in \cB(M/i)} \prod_{j \in B'}\frac{x_j}{1 - x_i} \\
        &\le (1 - x_i)^{r-1} \cdot \lambda(M/i). \qedhere
    \end{align*}
    \end{proof}

Using Lemma~\ref{lem: Lagrangian gradient} and Theorem~\ref{thm: Kung}, we bound the Lagrangian of a $U_{2,t+2}$-minor-free matroid for $t \ge 2$. Our techniques for optimizing the Lagrangian are well-known and standard. For examples, see \cite{MotzkinStraus, Sidorenko78} and the survey paper by Keevash \cite{Keevash}.

\begin{theorem} \label{thm: Lagrangian upper bound}
    Let $t \ge 2$ and let $r \ge 1$ be integers. Let $M$ be a rank-$r$ $U_{2,t+2}$-free matroid. Then
    \[ \lambda(M) \le b(r,t)\left(\frac{t-1}{t^r - 1}\right)^r.\]
\end{theorem}

\begin{proof}
    As $\lambda(M) = \lambda(\si(M))$, without loss of generality, we may suppose $M$ is simple. Let $s$ denote the number of elements of $M$. By  Theorem~\ref{thm: Kung}, $s \le  \frac{t^r - 1}{t - 1}$. Let $p(x_1,\ldots,x_s)$ be the Lagrangian polynomial of $M$. When $r = 1$, this clearly implies $\lambda(M) = 1$. If $s = r$, then by convexity we have $\lambda(M) = (1/r)^r$. Note then, 
    \begin{align*}
        \frac{r!}{r^r} = \prod_{i = 0}^{r-1} \left(1 - \frac{i}{r}\right) \le \prod_{i = 0}^{r-1} \left( 1 - \frac{t^i - 1}{t^r - 1} \right)
    \end{align*}
    where the inequality follows from the fact that the function $(t^x - 1)/x$ is increasing on $(0,\infty)$. By \eqref{eq:b_definition} and the above inequality, we have,
    \begin{align*}
        \lambda(M) = \left( \frac{1}{r} \right)^r \le \frac{1}{r!} \cdot \prod_{i = 0}^{r-1} \left( 1 - \frac{t^i - 1}{t^r - 1} \right) = \frac{1}{r!} \cdot \prod_{i = 0}^{r-1} \left(\frac{t^r - t^i}{t^r - 1} \right) = b(r,t)  \left(\frac{t-1}{t^r - 1} \right)^r.
    \end{align*}

    We proceed with induction on $r$ and $s$. Let $\mathbf{x}^* = (x_1^*,\ldots,x_s^*)$ maximize $p(x_1,\ldots,x_s)$ over the standard simplex. If $x^*_i = 0$, then the claim holds by induction on $s$. So we may suppose $x^*_i > 0$ for all $i \in [s]$. Thus  $\mathbf{x}^*$ belongs to the relative interior of the standard simplex, and as a consequence, the gradient of $p$ must be orthogonal to the hyperplane defined by $\sum_i x_i = 1$. In particular, there exists a constant $c \in \mathbb{R}$ such that $\frac{\partial}{\partial x_i} p(\mathbf{x}^*) = c$ for all $i \in [s]$. By Euler's identity for homogeneous functions, we have 
    \begin{align*}
            r \cdot p(\mathbf{x}^*) = \sum_{i = 1}^s x_i^* \frac{\partial}{\partial x_i} p(\mathbf{x}^*) = \sum_{i = 1}^s x_i^* c = c.
    \end{align*}
    By Lemma~\ref{lem: Lagrangian gradient}, for all $i \in [s]$, \[ c = \frac{\partial}{\partial x_i} p\left(\mathbf{x}^*\right) \le (1 - x_i^*)^{r-1} \lambda(M/i) = (1 - x_i^*)^{r-1}\lambda(\si(M/i)).\]  
    By the pigeonhole principle, there exists some $j$ such that $x_j \ge \frac{t-1}{t^r - 1}$. It follows by induction on $r$ and (\ref{eq:b_recursion}) that
    \begin{align*}
        c &\le \left(1 - \frac{t-1}{t^r - 1}\right)^{r-1} b(r-1,t) \left(\frac{t-1}{t^{r-1} - 1} \right)^{r-1} \\
        &= \left( \frac{t(t^{r-1} - 1)}{t^r - 1} \right)^{r-1} b(r-1,t) \left(\frac{t-1}{t^{r-1} - 1} \right)^{r-1}\\
        &= \left( \frac{t(t^{r-1} - 1)}{t^r - 1} \right)^{r-1}b(r,t) \frac{r (t - 1)}{t^{r-1}(t^r - 1)} \left(\frac{t-1}{t^{r-1} - 1} \right)^{r-1} = r\cdot b(r,t)\left(\frac{t-1}{t^r - 1} \right)^r.
    \end{align*}
    We have $r \cdot p(\mathbf{x}^*) \le r\cdot b(r,t)\left((t-1)/(t^r - 1) \right)^r$, which yields the claim.
\end{proof}

We comment that the computation for the Lagrangian for binary matroids was done implicitly by Alon \cite{Alon} via probabilistic techniques in the context of  bounding the size of erasure list-decodable codes. A classic result of Tutte~\cite{Tutte} states a matroid is binary (i.e.\ respresentable over $\bF_2$) if and only if it is $U_{2,4}$-minor free. Shortly after the submission of this manuscript, Ellis, Ivan, and Leader~\cite[Corollary 3]{EIL25} proved the same bound for $\bF_t$-representable matroids in the special case that $t$ is a prime power. It is well-known that $\bF_t$-representable matroids have  no $U_{2,t+2}$-minor, see~\cite[Corollary~6.5.3]{Oxley2011}.
Rota famously conjectured that for every finite field $\bF$ there is a finite list of excluded minors for $\bF$-representability \cite{Rota}, and a proof of this conjecture was announced by Geelen, Gerards, and Whittle \cite{GeelenGerardsWhittle2014} but is still in the process of being written.

We have the following corollary for the maximum number of bases of an $n$-element, rank-$r$ matroid with no $U_{2,t+2}$-minor.

\begin{theorem} \label{thm: finite field daisy matroid Turan density}
Let $n, r,t$ be integers with $r\ge 1$ and $t\ge 2$.Then
\[ \ex_{\M}(n,r,U_{2,t+2}) \le b(r, t) \cdot \left( \frac{n(t-1)}{t^r - 1} \right)^r.\]
 Furthermore, this bound is tight when $t$ is a prime power and $n$ is a multiple of $\frac{t^r - 1}{t-1}$.  
\end{theorem}
\begin{proof}
Let $M$ be an $n$-element rank-$r$ matroid with no $U_{2,t+2}$-minor, and let $N = \si(M)$.
For each element $i$ of $N$, let $y_i$ be the number of elements of $M$ parallel to $i$, and let $x_i = y_i/n$.
Then 
\begin{align*}
    b(M) = \sum_{B \in \cB(N)} \prod_{i \in B}y_i = n^r \cdot \sum_{B \in \cB(N)} \prod_{i \in B}x_i \le n^r \cdot \lambda(N) \le n^r \cdot b(r, t) \cdot \left( \frac{t-1}{t^r - 1} \right)^r\!\!,
\end{align*}
where the last inequality holds by Theorem \ref{thm: Lagrangian upper bound}.
When $n$ is a multiple of $\frac{t^r - 1}{t - 1}$ and $t$ is a prime power, equality holds for the loopless $n$-element rank-$r$ matroid whose simplification is $\PG{r-1}{t}$ and whose parallel classes all have equal size.
\end{proof}

Since Theorem \ref{thm: finite field daisy matroid Turan density} is asymptotically tight whenever $t$ is a prime power, we have the following corollary.

\begin{corollary} \label{cor: density for D(2, q+2)}
For every integer $r \ge 2$ and every prime power $q$ we have $$\pi_{\M}(r,U_{2,q+2}) = b(r,q) \cdot \left(\frac{q-1}{q^r - 1}\right)^r \cdot r! = \prod_{i = 1}^{r-1}\left(1 - \frac{q^i-1}{q^r - 1}\right).$$
\end{corollary}

In particular, this implies that the lower bound of Theorem \ref{thm: ellis-ivan-leader prime power lower bound} is tight when restricted to matroid basis hypergraphs.

\begin{corollary} \label{cor: limit r to infty for D(2, q+2)}
For every prime power $q$ we have 
$$\lim_{r\to\infty}\pi_{\M}(r, U_{2, q+2}) = \prod_{i = 1}^{\infty}(1 - q^{-i}).$$
\end{corollary}

While Theorem~\ref{thm: finite field daisy matroid Turan density} is tight when $t$ is a prime power, this is likely not the case when $t$ is not a prime power because of the following result.

\begin{theorem}[Geelen and Nelson \cite{GeelenNelson2010}] \label{thm: Geelen-Nelson}
Let $t \ge 2$ be an integer, and let $q$ be the largest prime power at most $t$.
For all sufficiently large $r$, every simple rank-$r$ matroid with no $U_{2,t+2}$-minor has at most $\frac{q^r - 1}{q - 1}$ elements.
\end{theorem}

In light of this theorem and Theorem \ref{thm: finite field daisy matroid Turan density}, we make the following conjecture.

\begin{conjecture} \label{conj: general D_r(2,t)}
Let $t \ge 2$, and let $q$ be the largest prime power at most $t$.
For all sufficiently large integers $r$ and $n$ we have 
$$\ex_{\M}(n,r,U_{2,t+2}) \le b(r, q) \cdot \left( \frac{n(q-1)}{(q^r - 1)} \right)^r.$$
\end{conjecture}

If true, Conjecture~\ref{conj: general D_r(2,t)} would be asymptotically tight. Conjecture \ref{conj: general D_r(2,t)} seems more difficult than Theorem \ref{thm: finite field daisy matroid Turan density}.
The main difficulty is that we can no longer use direct induction on $r$ because Theorem~\ref{thm: Geelen-Nelson} only applies when $r$ is sufficiently large.
In spite of this, we can use standard prime estimation techniques to prove an approximate version of Conjecture~\ref{conj: general D_r(2,t)}; a similar estimate was made in \cite{EllisIvanLeader}. We use the following well-known prime estimation result.

\begin{theorem}[Baker, Harman, and Pintz~\cite{BakerHarmanPintz2001}] \label{thm:prime_estimate_below_t}
    For all sufficiently large real numbers $x$, there is a prime in the interval $[x - x^{0.525}, x]$.
\end{theorem}

We can now approximate $\ex_M(n,r,U_{2,t+2})$ using the largest prime at most~$t$.

\begin{theorem}
    Let $r$ be a fixed positive integer and let $t$ be a sufficiently large integer. For all $n \ge r$ we have
    \[\ex_{\M}(n,r,U_{2,t+2}) \le \left( b(r,q) \left((q-1)/(q^r - 1)\right)^r + O_t(t^{\theta-2})\right)n^r\] where $\theta = 0.525$ and $q$ is the largest prime at most $t$. 
\end{theorem}
\begin{proof}
Let $t_0$ be such that the interval $[t-t^{\theta}, t]$ contains a prime number for all $t \ge t_0$ (such a $t_0$ exists by Theorem~\ref{thm:prime_estimate_below_t}), and assume that $t \ge t_0$.
As $t \le q + t^{\theta}$, we have $\ex_{\M}(n, U_{2,t+2}) \le \ex_{\M}(n, U_{2,q+t^{\theta}+2})$ because every matroid with no $U_{2, t+2}$-minor also has no $U_{2, q+t^{\theta}+2}$-minor. By \eqref{eq:b_definition},
\begin{align*}
    b(r,t) \left( (t-1)/(t^r - 1) \right)^r = \frac{1}{r!} \prod_{i = 0}^{r-1} \left (\frac{t^r - t^i}{t^r - 1}\right) = \frac{1}{r!} \prod_{i = 0}^{r-1} \left( 1 - \frac{t^i - 1}{t^r - 1} \right).
\end{align*}
As $q \le t$, by the binomial theorem we have that
$$ (q + t^{\theta})^r \le q^r + 2^rt^{r + \theta - 1}.$$
As $t$ is sufficiently large, there exists a constant $C = C(r)$,
\begin{align*}
    -\frac{1}{(q + t^{\theta})^r - 1} \le -\frac{1}{q^r + 2^rt^{r + \theta - 1} - 1} \le -\frac{1}{q^r - 1} + \frac{C}{t^{r + 1 -\theta}}.
\end{align*}

We have, 
\begin{align*}
    \frac{1}{r!} \prod_{i = 0}^{r-1} \left( 1 - \frac{t^i - 1}{t^r - 1} \right) &\le \frac{1}{r!} \prod_{i = 0}^{r-1} \left( 1 - \frac{t^i - 1}{(q + t^{\theta})^r - 1} \right) \le \frac{1}{r!} \prod_{i = 0}^{r-1} \left(1 -\frac{t^i-1}{q^r - 1} + \frac{C(t^i-1)}{t^{r+1 - \theta}}\right)\\
    &\le \frac{1}{r!} \prod_{i = 0}^{r-1} \left(1 -\frac{q^i-1}{q^r - 1} + \frac{C(t^i-1)}{t^{r+1-\theta}}\right) = \frac{1}{r!} \prod_{i = 0}^{r-1} \left(1 -\frac{q^i-1}{q^r - 1}\right) + O_t\left(\frac{1}{t^{2-\theta}}\right).
\end{align*}
By Theorem \ref{thm: finite field daisy matroid Turan density}, the claim follows.
\end{proof}

Once again, we have a corollary for Tur\'an basis density.

\begin{corollary} \label{cor: density for D_r(2,t+2)}
Let $r$ and $t$ be integers with $r \ge 2$ and $t$ sufficiently large, and let $q$ be the largest prime power at most $t$. Then
$$\prod_{i = 1}^{r-1}\Big(1 - \frac{q^i-1}{q^r - 1}\Big) \le \pi_{\M}(r,U_{2,t+2}) 
\le \prod_{i = 1}^{r-1}\Big(1 - \frac{q^i-1}{q^r - 1}\Big) + \frac{C \cdot r!}{t^{2-\theta}},$$
where $\theta = 0.525$ and $C$ is a positive constant depending only on $r$.
\end{corollary}

\section{$U_{s, s+1}$} \label{sec: circumference}
In this section we consider $U_{s, s+1}$.
The problem of finding $\ex_{\M}(n,r,U_{s, s+1})$ has a particularly interesting interpretation, which we briefly describe.
First, the \emph{circumference} of a matroid with at least one circuit is the number of elements in a largest circuit, and a matroid has circumference at most $s$ if and only if it has no $U_{s,s+1}$-minor. 
To see this, let $M$ be a matroid whose largest circuit has $k \ge s+1$ elements; deleting all elements outside the circuit and then contracting $k-s-1$ elements from the circuit results in a $U_{s,s+1}$-minor. In the other direction, note that the circumference cannot increase when deleting or contracting an element, so if $M$ has a $U_{s,s+1}$-minor, then the circumference of $M$ is at least the circumference of $U_{s,s+1}$, which is $s+1$.
So, $\ex_{\M}(n,r, U_{s,s+1})$ is the maximum number of bases over all matroids with $n$ elements, rank $r$, and circumference at most $s$.

When $s = 1$ we have $\ex_{\M}(n,r,U_{s,s+1}) = 1$ by Proposition \ref{prop: s = 1}.
We can easily get a tight upper bound when $s = 2$.

\begin{proposition} \label{prop: D(2,3)}
Let $n \ge r \ge 2$. Then 
\[ \ex_{\M}(n,r,U_{2,3}) \le \left(\frac{n}{r}\right)^ r,\] with equality when $r$ divides $n$.
\end{proposition}
\begin{proof}
Let $n \ge r \ge 2$ be integers, and let $M$ be an $n$-element rank-$r$ matroid with no $U_{2,3}$-minor. Without loss of generality, we may assume that $M$ is loopless. Let $B$ be a basis of $M$.
Every element of $M$ is parallel to an element of $B$, or else $M$ has a $U_{2,3}$-minor.
So $M$ is a direct sum of parallel classes $P_1, P_2, \dots, P_r$ with $\sum_i |P_i| = n$.
The number of bases of $M$ is $\prod_i |P_i|$, and since $\sum_i |P_i| = n$, by convexity, it follows that $\prod_i |P_i| \le (\frac{n}{r})^r$, as desired.
When $r$ divides $n$, equality holds when $|P_i| = n/r$ for all $i \in [r]$.
\end{proof}

We next consider the case when $s = 3$. The following lemma is a structural result for $U_{3,4}$-minor-free matroids.

\begin{lemma}\label{lem: D_r(3,4) structure}
    If $M$ be a $U_{3,4}$-minor-free matroid, then $M$ is a direct direct sum of matroids of rank at most 2. 
\end{lemma}
\begin{proof}
    The claim is trivial if $M$ has rank at most $2$. Suppose $M$ has rank $r \ge 3$, and without loss of generality, suppose $M$ is simple. 
    Let $B$ be a basis of $M$. Then each element not in $B$ is spanned by a pair from $B$, or else $M$ has a $U_{3,4}$-minor.
    If $e,f \in E(M) - B$ are spanned by the pairs $\{b_1, b_2\}$ and $\{b_2, b_3\}$, respectively, with $b_1 \ne b_3$, then by the circuit elimination axiom there is a circuit $C$ contained in $\{b_1, b_3, e, f\}$.
    Then $C$ must be $\{b_1, b_3, e, f\}$ or else $M$ has two distinct lines that share two points, a contradiction.
    Therefore there is a partition of $B$ so that each part has size at most 2, each $e \in E(M) - B$ is spanned by a part of size 2, and each part of size 2 spans some $e \in E(M) - B$. These parts span matroids whose direct sum is equal to $M$, so $M$ is a direct sum of matroids of rank at most 2.
\end{proof}

For each real number $x$ and positive integer $k$ we write 
$$\binom{x}{k} = \prod_{i=0}^{k-1}\frac{x - i}{k - i}.$$
Recall that for a matroid $M$ we write $b(M)$ for the number of bases of $M$.

\begin{theorem}\label{thm: D_r(3,4)}
Let $n$ and $r$ be positive integers with $n > r \ge 2$.
\begin{enumerate}[$(i)$]
\item If $r$ is even, then $\ex_{\M}(n,r,U_{3,4}) \le \binom{\frac{2n}{r}}{2}^{\frac{r}{2}}$, and $\pi_{\M}(r,U_{3,4}) = \frac{r! \cdot 2^{r/2}}{r^r}$.

\item If $r$ is odd, then $\ex_{\M}(n,r,U_{3,4}) \le \big(1 + o_n(1)\big)\frac{n}{r}\binom{\frac{2n}{r}}{2}^{\frac{r-1}{2}}$ and $\pi_{\M}(r,U_{3,4}) = \frac{r! \cdot 2^{\frac{r-1}{2}}}{r^r}$.
\end{enumerate}
\end{theorem}
\begin{proof}
Let $M$ be an $n$-element, rank-$r$, $U_{3,4}$-minor-free matroid such that $b(M) = \ex_{\M}(n,r,U_{3,4})$. By Lemma~\ref{lem: D_r(3,4) structure}, we may assume that $M$ is a direct sum of matroids of rank at most two. We first show that the desired upper bounds hold by induction on $r$.
When $r = 2$, clearly $(i)$ holds, so we may assume that $r \ge 3$.

We consider two cases, depending on the parity of $r$.
First  suppose that $r$ is even. 
Then $M = M_1 \oplus M_2$ where $r(M_1) = 2$ and $r(M_2) = r - 2$.
Let $n_i = |M_i|$ for $i = 1,2$, so $n_1 + n_2 = n$.
Then 
\[b(M) = b(M_1) \cdot b(M_2) \le \binom{n_1}{2} \cdot \binom{\frac{2n_2}{r-2}}{2}^{\frac{r-2}{2}},\]
where the bound on $b(M_2)$ holds by induction on $r$.
Since $n_1 + n_2 = n$, this product is maximized when $n_1 = \frac{2n}{r}$ and $n_2 = n - \frac{2n}{r}$, so $n$ is distributed equally to each of the $\frac{r}{2}$ factors in the product.
Since $\frac{2n_2}{r-2} = \frac{2n}{r }$, we see that $b(M) \le \binom{\frac{2n}{r}}{2}^{\frac{r}{2}}$, as desired.

Next suppose $r$ is odd.
Then $M = M_1 \oplus M_2$, where $r(M_1) = 1$ and $r(M_2) = r - 1$.
Let $n_i = |M_i|$ for $i = 1,2$, so $n_1 + n_2 = n$.
Then 
\[b(M) = b(M_1) \cdot b(M_2)\le n_1 \cdot \binom{\frac{2n_2}{r-1}}{2}^{\frac{r-1}{2}} \le \left( \frac{\sqrt{2}}{r-1} \right)^{r-1}n_1n_2^{r-1},\]
where the bound on $b(M_2)$ holds by induction on $r$.
Since $n_1 + n_2 = n$, this product is maximized when $n_1 = \frac{n}{r}$ and $n_2 = \frac{r-1}{r}n$. It follows that $b(M) \le (\sqrt{2})^{r-1} (n/r)^r \le \big(1 + o_n(1)\big)\frac{n}{r}\binom{\frac{2n}{r}}{2}^{\frac{r-1}{2}}$, as desired.

We now combine the upper bounds with asymptotically matching lower bounds to find $\pi_{\M}(r,U_{3,4})$.
If $r$ is even and $r/2$ divides $n$, let $M$ be the direct sum of $r/2$ copies of $U_{2, \frac{n}{r/2}}$.
Then $b(M) = \binom{\frac{2n}{r}}{2}^{\frac{r}{2}}$ and $M$ has no $U_{3,4}$-minor.
Combined with the upper bound from $(i)$ we see that 
$$\pi_{\M}(r,U_{3,4}) = \lim_{n \to \infty}\frac{\binom{\frac{2n}{r}}{2}^{\frac{r}{2}}}{\binom{n}{r}} = \frac{r! \cdot 2^{r/2}}{r^r}.$$
If $r$ is odd and $r$ divides $n$, let $M$ be the direct sum of $\frac{r-1}{2}$ copies of $U_{2, \frac{2n}{r}}$ and one copy of $U_{1, \frac{n}{r}}$.
Then $b(M) = \frac{n}{r}\binom{\frac{2n}{r}}{2}^{\frac{r-1}{2}}$ and $M$ has no $U_{3,4}$-minor.
Combined with the upper bound from $(ii)$ we see that 
\begin{equation*}
    \pi_{\M}(r,U_{3,4}) = \lim_{n \to \infty}\frac{\frac{n}{r}\binom{\frac{2n}{r}}{2}^{\frac{r-1}{2}}}{\binom{n}{r}} = \frac{r! \cdot 2^{\frac{r-1}{2}}}{r^r}.\qedhere
\end{equation*}
\end{proof}

We have the following corollary for $r = 3$.

\begin{corollary} \label{cor: D_3(3,4)}
$\pi_{\M}(3,U_{3,4}) = 4/9$.
\end{corollary}

The basis hypergraph of $U_{3,4}$ is the 3-uniform hypergraph $K_4^{(3)}$, and while its Tur\'an density is not known, a construction due to Tur\'an \cite{Turan} shows that $\pi(K_4^{(3)}) \ge 5/9$.
So Corollary~\ref{cor: D_3(3,4)} is notable because it shows that the Tur\'an density of $K_4^{(3)}$ is significantly smaller when we restrict to matroid basis hypergraphs.

We next consider $s \ge 4$.
Matroids with circumference $s$ are not as well-structured when $s \ge 4$, but we expect that $\ex_{\M}(n,r,U_{s,s+1})$ is achieved by a direct sum of uniform matroids of rank $s - 1$, as it is for $s = 2$ and $s = 3$.

\begin{conjecture} \label{conj: D(s,s+1) with s >= 4}
Let $s,r$ be positive integers with $s \ge 4$ and $r \ge s - 1$.
Then $\ex_{\M}(n,r,U_{s,s+1}) \le \binom{\frac{n}{r/(s-1)}}{s-1}^{\frac{r}{s-1}}$ for all $n \ge r$, with equality when $s - 1$ divides $r$ and $\frac{r}{s-1}$ divides $n$ for the direct sum of $r/(s-1)$ copies of $U_{s-1, \frac{n}{r/(s-1)}}$.
\end{conjecture}

This may not be difficult to prove, but the direct structural approach of Theorem~\ref{thm: D_r(3,4)} seemingly does not apply.

\section{$U_{3,t}$} \label{sec: D(3,t)}

After considering $\pi_{\M}(r,U_{s,t})$ when $s \in \{1,2\}$ or $t = s + 1$, perhaps the next natural cases to consider are $s = 3$ or $t = s + 2$.
While we believe that $\pi_{\M}(r,U_{s,s+2})$ and $\pi_{\M}(r,U_{3,t})$ are interesting parameters for future work, we will next focus on the simpler parameter $\pi_{\M}(3,U_{3,t})$, the Tur\'an basis density for rank-3 matroids with no $U_{3,t}$-restriction.
We freely use the facts that any two lines in a matroid share at most one point, and that if $\cL$ is a collection of subsets of a set $E$ so that each set in $\cL$ has at least three elements and each pair of sets in $\cL$ shares at most one common element, then $\cL$ is the set of long lines (i.e.\ lines with at least three points) of a simple rank-3 matroid with ground set $E$ \cite[Proposition~1.5.6]{Oxley2011}.

We first give lower bounds for $\pi_{\M}(3, U_{3,t})$.

\begin{proposition} \label{prop: lower bounds for D_3(3,t)}
For all integers $m \ge 2$, 
\begin{enumerate}[$(a)$]
    \item $\pi_M(3,U_{3,2m+1}) \ge 1 - \frac{1}{m^2}$ and

    \item $\pi_M(3, U_{3,2m+2}) \ge \frac{4m^4}{(2m^2+1)^2} = 1 - \frac{2}{2m^2+1} + \frac{1}{(2m^2+1)^2}$.
\end{enumerate}
\end{proposition}
\begin{proof}
    For (\textit{a}), it suffices to exhibit an $n$-element rank-3 matroid that does not have a $U_{3,2m+1}$-restriction and has a large number of bases.
    Suppose that $m$ divides $n$, and let $M$ be the simple rank-3 matroid on ground set $\{1,2,\ldots,n\}$ with long lines $L_1, L_2, \ldots, L_m$ where $L_i = \{m(i-1)+1, m(i-1)+2, \ldots, mi\}$. Any set of $2m+1$ points intersects one of these lines in at least three points, so $M$ does not have a $U_{3,2m+1}$-restriction.
    The bases of $M$ are all 3-elements subsets that are not contained in one of the $L_i$, so
    \begin{equation*}
        b(M) = \binom{n}{3} - m\binom{n/m}{3} = \binom{n}{3}\left(1-\frac{1}{m^2}\right) + o(n^3).
    \end{equation*}

    For (\textit{b}), it suffices to exhibit an $n$-element rank-3 matroid that does not have a $U_{3,2m+2}$-restriction and has a large number of bases.
    Let $\alpha = \frac{2m}{2m^2+1}$ and let $\beta = 1 - m\alpha$. Let $E$ be an $n$-element set, and partition $E = L_1 \cup L_2 \cup \ldots \cup L_m \cup P$ such that $\big||L_i|-\alpha n\big| \le 1$ for all $i$ and $\big||P|-\beta n\big|\le 1$. Let $M$ be the rank-3 matroid on $E$ in which $L_1, L_2, \ldots, L_m$ are the long lines of $M$, the restriction $M|L_i$ is simple for all $i \in [m]$, and $P$ is a parallel class.
    By the pigeonhole principle, each set of $2m+2$ elements intersects $P$ in more than one element or one of the sets $L_1, L_2, \ldots, L_m$ in more than two elements, so $M$ does not have a $U_{3,2m+2}$-restriction.
    The number of bases of $M$ is
    \begin{align*}
            b(M)
            &= |P|\binom{\sum_i |L_i|}{2} + \sum_{1 \le p < q < r \le m} |L_p||L_q||L_r| + \sum_{\substack{1 \le p, q \le m\\ p\neq q}} \binom{|L_p|}{2}|L_q| \\
            &= \beta n \binom{m\alpha n}{2} + \binom{m}{3}(\alpha n)^3 + m(m-1)\binom{\alpha n}{2}\alpha n + o(n^3) \\
            &= \binom{n}{3}\left(\frac{2m^2}{2m^2+1}\right)^2 +o(n^3). \qedhere
    \end{align*}
\end{proof}

We next prove a general structural result for rank-$3$ matroids with no $U_{3,t}$-restriction (in two cases, depending on the parity of $t$).
Roughly speaking, this theorem says that every matroid with no $U_{3, 2m+1}$-restriction is a union of at most $m$ lines and a small leftover set.

\begin{theorem} \label{thm: structure for D(3,2m+1)}
Let $m \ge 2$.
If $M$ is a rank-$3$ matroid with no $U_{3,2m+1}$-restriction, then there is an integer $k$ with $0 \le k \le m$ and disjoint sets $X$ and $Y$ with $X \cup Y = E(M)$ such that 
\begin{itemize}
    \item $M|X$ is a union of $k$ lines, and

    \item $M|Y$ has no $U_{3, 2(m-k)+1}$-restriction (if $k < m$), no $U_{2, \binom{2(m-k)}{2}+2}$-restriction, and at most $\binom{2(m-k)}{2}\big(\binom{2(m-k)}{2} - 1\big) + 2(m-k)$ points.
\end{itemize}
\end{theorem}
\begin{proof}
We may assume that $M$ is simple.
Let $(L_1, L_2, \dots, L_k)$ be a maximal sequence of subsets of $E(M)$ so that $k \le m$ and for each $i\in [k]$ the set $L_{i}$ is a line of $M\del (L_1 \cup \cdots \cup L_{i-1})$ with at least $\binom{2(m-i)+1}{2}+2$ points (taking $L_0 = \varnothing$). 
Let $X = L_1\cup L_2 \cup \cdots \cup L_k$ and let $Y = E(M) - X$.
By maximality, each line of $M|Y$ has at most $\binom{2(m-(k+1))+1}{2} + 1 = \binom{2(m-k)}{2} + 1$ points.

Suppose $M|Y$ has a $U_{3,2(m-k)+1}$-restriction, where $U_{3,1}$ refers to a single point. Then $k \ge 1$ because $M$ has no $U_{3,2m+1}$-restriction. The lines $L_1, \dots, L_k$ are large enough that we can iteratively (in reverse order) add two points from each line to extend the $U_{3,2(m-k)+1}$-restriction of $M|Y$ to a $U_{3, 2m+1}$-restriction of $M$.
We now make this idea precise.
Write $Y = L_{k+1}$ for convenience. Let $j$ be minimum so that $M$ has a $U_{3, 2(m-j+1) + 1}$-restriction $N$ contained in $L_{j}~\cup~L_{j+1}~\cup~\cdots~\cup~L_{k+1}$.
Note that $j > 1$ because $M$ has no $U_{3,2m+1}$-restriction, and that $j$ is well-defined because $j = k+1$ is a valid option.
The points of $N$ span $\binom{2(m-j+1) + 1}{2}$ lines, and each line contains at most one point of $L_{j-1}$.
Since $L_{j-1}$ has at least $\binom{2(m-(j-1))+1}{2} + 2$ points, it contains points $a$ and $b$ not spanned by any line of $N$.
Also, $\{a,b\}$ does not span any point $e$ of $N$, because then $e \in L_{j-1}$, while $N$ is disjoint from $L_{j-1}$ by the definitions of $(L_1, L_2, \dots, L_k)$ and $Y$.
Therefore $N \cup \{a,b\}$ is a $U_{3, 2(m-j+1) +3}$-restriction of $M$ contained in 
$L_{j-1}~\cup~L_{j}~\cup~\cdots~\cup~L_{k+1}$, which contradicts the minimality of $j$.
Therefore $M|Y$ has no $U_{3,2(m-k)+1}$-restriction. 

Finally, we bound the number of points of $M|Y$. If $Y$ is a set of rank at most 2, which is true in particular when $k=m$ or $k=m-1$, then the bound follows by a straightforward computation. So we may assume that $Y$ has rank 3. Let $t$ be maximum so that $M|Y$ has a $U_{3,t}$-restriction $N$.
By the maximality of $t$, each point of $M|Y$ is spanned by a pair of elements of $N$.
So $M|Y$ is a union of $\binom{t}{2}$ lines, each with at most $\binom{2(m-k)}{2} - 1$ points not in $E(N)$.
Since $t \le 2(m-k)$, it follows that $M|Y$ has at most $\binom{2(m-k)}{2}(\binom{2(m-k)}{2} - 1) + 2(m-k)$ points.
\end{proof}

Note that if $k = m$ then $M|Y$ has $0$ points, and if $k = m - 1$ then $M|Y$ has no $U_{3,3}$-restriction and therefore has rank at most $2$.
So if $k \in \{m-1, m\}$, then $M$ is covered by $m$ lines.
In particular, when $m = 2$, Theorem \ref{thm: structure for D(3,2m+1)} implies that every rank-$3$ matroid with no $U_{3,5}$-restriction is either the union of two lines, or has at most $34$ points.

We have a similar result for matroids with no $U_{3,2m+2}$-restriction; we omit the proof as it is nearly identical to the proof of Theorem \ref{thm: structure for D(3,2m+1)}.

\begin{theorem} \label{thm: structure for D(3,2m+2)}
Let $m \ge 2$.
If $M$ is a rank-$3$ matroid with no $U_{3,2m+2}$-restriction, then there is an integer $k$ with $0 \le k \le m$ and disjoint sets $X$ and $Y$ with $X \cup Y = E(M)$ so that 
\begin{itemize}
    \item $M|X$ is a union of $k$ lines, and

    \item $M|Y$ has no $U_{3, 2(m-k)+2}$-restriction (if $k < m$), no $U_{2, \binom{2(m-k)}{2}+2}$-restriction, and at most $\binom{2(m-k)+1}{2}\big(\binom{2(m-k)+1}{2} - 1\big) + 2(m-k)+1$ points.
\end{itemize}
\end{theorem}

Note that if $k = m$ then $M|Y$ has no $U_{2,2}$-restriction and therefore consists of at most one point, and if $k = m - 1$ then $M|Y$ has no $U_{3,4}$-restriction and is therefore the direct sum of a line and a point by Lemma \ref{lem: D_r(3,4) structure}.
So if $k \in \{m-1, m\}$, then $M$ is covered by $m$ lines and a point.

Roughly speaking, Theorems \ref{thm: structure for D(3,2m+1)} and \ref{thm: structure for D(3,2m+2)} reduce the problem of finding $\pi_{\M}(3, U_{3,t})$ to analyzing matroids with a small number of points.
We do this for $t = 5$.

\begin{lemma} \label{lemma: structure for D_3(3,5)}
Let $M$ be a rank-$3$ matroid with no $U_{3,5}$-restriction. Either
\begin{enumerate}[$(a)$]
\item $M$ has no $U_{2,5}$-minor, or 

\item $M$ is the union of two lines.
\end{enumerate}
\end{lemma}
\begin{proof}
We may assume that $M$ is simple, because the statement holds for $M$ if and only if it holds for $\si(M)$. 
An ordered pair $(L_1, L_2)$ of lines of $M$ is \emph{good} if $|L_1 - L_2| \ge 4$ and $|L_2 - L_1| \ge 2$.

First suppose that $M$ has a good pair $(L_1, L_2)$ of lines.
Suppose there is a point $w$ of $M$ outside of $L_1 \cup L_2$.
Let $x,y \in L_2 - L_1$.
The lines $\cl_M(\{w,x\})$ and $\cl_M(\{w,y\})$ each contain at most one point in $L_1 - L_2$, and so there are points $s,t \in L_1 - L_2$ that are not spanned by any pair of $\{w,x,y\}$.
But then $\{w,x,y,s,t\}$ forms a $U_{3,5}$-restriction of $M$, a contradiction.
So $E(M) \subseteq L_1 \cup L_2$, and therefore $(b)$ holds when $M$ has a good pair of lines.

Next suppose that $M$ has a line $L_1$ with $|L_1| \ge 5$.
There are points $x$ and $y$ of $M$ outside of $L_1$, or else $(b)$ holds.
Let $L_2 = \cl_M(\{x,y\})$.
Then $(L_1, L_2)$ is a good pair of lines of $M$, and so $(b)$ holds.

So we may assume that $M$ has no good pairs and no lines with at least five points.
We will show that $(a)$ holds. Suppose, towards a contradiction, that $M$ has a $U_{2,5}$-minor.
Since $M$ has a $U_{2,5}$-minor, there is a point $e_0$ so that $M/e_0$ has a $U_{2,5}$-restriction.
Then $e_0$ is on at least five lines of $M$, because the lines of $M$ through $e_0$ correspond to the points of $M/e_0$. 
Let $T$ be a transversal of five of the lines through $e_0$.
Since $M$ has no line with at least five points, $M|T$ has rank $3$.
And since $M|T$ is a simple rank-$3$ matroid with five points, it has a $U_{3,4}$-restriction or contains a line with four points. In the former case, adding $e_0$ to the $U_{3,4}$-restriction gives $U_{3,5}$, because no pair from $T$ spans $e_0$, and this is a contradiction. In the latter case, $(L_1, L_2)$ is a good pair of lines, where $L_1 \subseteq T$ is the line with four points and $L_2 = \cl_M((T\setminus L_1) \cup \{e_0\})$, and this is again a contradiction.
\end{proof}

We can use the previous lemma to find $\pi_{\M}(n,3, U_{3,5})$.
Recall that $\ex_{\M}(n,3,U_{2,5}) \le b(3, 3) \cdot ( \frac{n}{13} )^3$, by Theorem \ref{thm: finite field daisy matroid Turan density}, where $b(3,3)$ is the number of bases of the ternary projective plane, which is $\frac{13 \cdot 12 \cdot 9}{3!} = 234$.

\begin{theorem} \label{thm: D_3(3,5)}
If $n \ge 14$, then
\[\ex_{\M}(n,3,U_{3,5}) \le \left(\frac{n}{2}\right)^3 - \left(\frac{n}{2}\right)^2.\] 
Furthermore, equality holds when $n$ is even, so $\pi_{\M}(3, U_{3,5}) = 3/4$.
\end{theorem}
\begin{proof}
Let $M$ be an $n$-element, rank-$3$ matroid with no $U_{3,5}$-restriction and $b(M) = \ex_{\M}(n,r,U_{3,5})$.
Either $(a)$ or $(b)$ holds from Lemma \ref{lemma: structure for D_3(3,5)}.
If $(a)$ holds, then $b(M) \le \ex_{\M}(n,r, U_{2,5})$ and therefore by Theorem~\ref{thm: finite field daisy matroid Turan density},
\[ \ex_{\M}(n,3,U_{3,5}) \le \ex_{\M}(n,3,U_{2,5}) = b(3,3) \Big(\frac{n}{13}\Big)^3 = 234\Big(\frac{n}{13}\Big)^3 <  \Big(\frac{n}{2}\Big)^3 - \Big(\frac{n}{2}\Big)^2,\]
where the last inequality holds because $n \ge 14$.
So $(b)$ holds, and $M$ is the union of lines $L_1$ and $L_2$.
It is straightforward to check that $b(M)$ is maximized when $L_1$ and $L_2$ are disjoint lines and have $\lfloor n/2\rfloor$ and $\lceil n/2 \rceil$ points, respectively.
Then $M$ has at most $(\frac{n}{2})^3 - (\frac{n}{2})^2$ bases, as desired.
\end{proof}

By Proposition \ref{prop: suspension_free_minor_free}, the basis hypergraph of a $U_{3,5}$-free matroid is $K^{(3)}_5$-subgraph-free. Tur\'{a}n~\cite{Turan1961} conjectured that $\pi(K^{(3)}_5) = 3/4$, and Theorem \ref{thm: D_3(3,5)} confirms this conjecture for matroid basis hypergraphs. The extremal construction of bases in the rank-$3$ matroid consisting of two lines is a special case of a family of examples due to Keevash and Mubayi, showing that $\pi(K^{(3)}_5) \ge 3/4$ -- see \cite[Section~9]{Keevash}. The actual value of $\pi(K^{(3)}_5)$ remains unknown.

\section{Directions for future work} \label{sec: open problems}
We close by discussing directions for future work relating to Problem \ref{prob: main problem}.

\subsection{Matroid truncations} \label{sec: truncations}
As discussed in Section~\ref{sec: matroid introducton}, some of the strongest constructions for daisy-free hypergraphs are naturally cast in matroidal terms. 
We believe that this connection can be further explored.
Ellis, Ivan, and Leader proved the following lower bound on $\pi(\cD_r(s,t))$.
\begin{theorem}[Ellis, Ivan, and Leader \cite{EllisIvanLeader}] \label{thm: Ivan-Ellis-Leader general lower bound}
Let $q$ be a prime power, $m$ a positive integer and $s$ even, and $t = \lfloor sq^{\frac{2m}{s}+1} + 1 \rfloor$.
Then
$$\lim_{r \to\infty} \pi(\cD_r(s,t)) \ge 1 - q^{-(m+1)} - O(q^{-(m+2)}).$$
\end{theorem}
In the special case that $s = m = k$ and $q = 2$, this shows that $\lim_{r \to\infty}\pi(\cD_r(k, 8k+1)) = 1 - O(2^{-k})$.
Improving on this, Alon \cite{Alon} showed that $\lim _{r \to \infty} \pi(\cD_r(k + 2 \log_2 k, 2k + 3\log_2 k)) \ge 1 - 2^{-k}$ where $k$ is some large fixed integer. Both of these results on link Tur{\'a}n densities are of interest due to applications to lower bounding vertex Tur{\'a}n densities of the hypercube, see \cite{EllisIvanLeader} for details. For recent progress regarding link Tur{\'a}n densities and vertex Tur{\'a}n problems in the hypercube, see \cite{Cameron,EllisIvanLeader_smallhypercubes}.

Just as in Theorem~\ref{thm: ellis-ivan-leader prime power lower bound}, both constructions are matroidal.
When $n = k(q^{m+r} - 1)$ for a positive integer $k$, they consider the $\bF_q$-representable matroid whose elements are the nonzero vectors in $\bF_q^{m+r}$ with $k$ copies of each vector.
(This matroid simplifies to the projective geometry $\PG{m+r-1}{q}$.)
They then take the $m$-th truncation of this matroid, where the \emph{$m$-th truncation} of a rank-$(r+m)$ matroid $M$ is the rank-$r$ matroid $T^{(m)}(M)$ whose bases are the rank-$r$ independent sets in $M$ \cite[Proposition 7.3.10]{Oxley2011}.
Then Theorem \ref{thm: Ivan-Ellis-Leader general lower bound} follows by showing that truncations do not lead to $U_{s,t}$-minors with $t$ arbitrarily large. Alon's bound is obtained by improving the analysis in their proof utilizing the classic Plotkin bound \cite{Plotkin} from the theory of error-correcting codes. We propose the following problem.
\begin{problem} \label{prob: projections of projective geometries}
Let $q$ be a prime power and let $r$ and $m$ be integers with $r \ge 2$ and $m \ge 0$.
For all $s$, what is the maximum $t$ so that $T^{(m)}(\PG{r+m-1}{q})$ has a $U_{s,t}$-minor?
\end{problem}

It will certainly be difficult to give an exact answer. Even when $m=0$, the problem is equivalent to the question over which fields the uniform matroid $U_{s,t}$ is representable~\cite[Problem~6.5.19]{Oxley2011}, an open problem in projective geometry and coding theory. However, as shown by Theorem~\ref{thm: Ivan-Ellis-Leader general lower bound}, even coarse bounds can have significant consequences for vertex Tur\'an problems on the hypercube.

\subsection{Discrete geometry} \label{sec: covering number}

By restricting the parameter $\ex_{\M}(n, s, U_{s,t})$ to affine matroids over $\bR$ (see \cite[Proposition 1.5.1]{Oxley2011}), we obtain the following natural problem.

\begin{problem} \label{prop: discrete geometry problem}
Let $n,s,t$ be positive integers with $t \ge s \ge 3$.
Among all sets of $n$ vectors in $\bR^{s-1}$ with no $t$ in general position, what is the maximum number of $s$-sets in general position?
\end{problem}

One can ask this question for vectors over other fields as well, but we focus on $\bR$ for simplicity.
In the special case $s = 3$, this problem is asking for the maximum number of non-collinear triples among $n$ points in $\bR^2$ with no $t$ in general position.
This is in some sense dual to the following problem of Erd\H os \cite{Erdos1988}: among all sets of $n$ points in $\bR^2$ with no four on a line, what is the maximum cardinality of a subset in general position, in the worst case?
See \cite{BaloghSolymosi} for recent progress.

By noting that the lower bound of Proposition~\ref{prop: lower bounds for D_3(3,t)} also holds for $\bR$-representable matroids, Theorem~\ref{thm: D_3(3,5)} solves Problem~\ref{prop: discrete geometry problem} in the case $(t,s) = (5, 3)$. 
We expect that the parameter $\ex_{\M}(n, s, U_{s,t})$ can be used to solve Problem~\ref{prop: discrete geometry problem} in other cases as well.

We mention one other direction for future work relating to Section \ref{sec: D(3,t)}.
For a positive integer $a$, the \emph{$a$-covering number} of a matroid $M$, denoted $\tau_a(M)$, is the minimum number of sets of rank at most $a$ required to cover the ground set of $M$.
In particular, $\tau_1(M)$ is the number of points of $M$, and $\tau_2(M)$ is the number of lines needed to cover the ground set of $M$.
Theorems \ref{thm: structure for D(3,2m+1)} and \ref{thm: structure for D(3,2m+2)} imply that if $M$ is a rank-$3$ matroid with no $U_{3,t}$-restriction, then $\tau_2(M) = O(t^2)$.
On the other hand, the $U_{3,t}$-free matroids exhibiting the bounds of Proposition \ref{prop: lower bounds for D_3(3,t)} have $2$-covering number $\Omega(t)$. 
This leads to the following problem.

\begin{problem} \label{prob: 2-covering number}
Find the maximum $2$-covering number for rank-$3$ matroids with no $U_{3,t}$-restriction.
\end{problem}

More generally, one could ask for the maximum $(s-1)$-covering number for rank-$s$ matroids with no $U_{s,t}$-restriction, but we will focus on $s = 3$ for simplicity.
Problem \ref{prob: 2-covering number} has several interesting variants and extensions.
It is interesting even if we restrict to the class of $\bR$-representable matroids: what is the maximum number of lines needed to cover $n$ points in $\bR^2$ with no $t$ in general position?
One could also allow rank to increase and ask for the maximum $2$-covering number for rank-$r$ matroids with no $U_{3,t}$-minor; this would be an analogue of Theorem \ref{thm: Geelen-Nelson}.

Problem \ref{prob: 2-covering number} also fits nicely with more general but coarser results of Geelen and Nelson \cite{GeelenNelson, Nelson} which show that there are very few possibilities for the behavior of $a$-covering numbers for matroids in minor-closed classes. 
We should also mention that while Problems \ref{prob: main problem} and \ref{prob: 2-covering number} ask for specific properties of the class of matroids with no $U_{s,t}$-minor, there are a number of deep conjectures about the overall structure of matroids in this class; see \cite{Geelen, GeelenGerardsWhittle}.

\subsection{Exact Tur\'{a}n basis numbers}

The bound on $\ex_\M(n,r,U_{2,t+2})$ in Theorem~\ref{thm: finite field daisy matroid Turan density} is tight whenever $t$ is a prime power and $n$ is a multiple of $\frac{t^r-1}{t-1}$; in this case the bound is attained by the matroid that is obtained from $\PG{r-1}{t}$ by replacing each element by a parallel class containing $n/\frac{t^r-1}{t-1}$ elements. In this section, we speculate on the exact value of $\ex_\M(n,r,U_{2,t+2})$ when $n < \frac{t^r-1}{t-1}$.

Let $t$ be a prime power, and let $r$ and $c$ be integers such that $1 \le c \le r-1$. Let $F$ be a rank-$(r-c)$ flat of $\PG{r-1}{t}$, and let $\BBG{r}{t}{c}$ be the matroid obtained by restricting $\PG{r-1}{t}$ to the complement of $F$. Note that $\BBG{r}{t}{c}$ has $\frac{t^r-t^{r-c}}{t-1}$ points and rank~$r$.

The matroids $\BBG{r}{t}{c}$ play a role in extremal matroid theory similar to that of balanced complete multipartite graphs in extremal graph theory: the matroid $\BBG{r}{t}{c}$ does not contain $\PG{c}{t}$ as a submatroid, and Bose and Burton~\cite{BoseBurton1966} showed that it is the densest submatroid of $\PG{r-1}{t}$ with this property.
We expect that $\BBG{r}{t}{c}$ has the most bases over all $\frac{t^r - t^{r-c}}{t-1}$-element, rank-$r$ matroids with no $U_{2, t+2}$-minor.

\begin{conjecture}\label{conj: bose_burton}
	Let $t \ge 2$ be prime power, and let $r$ and $c$ be integers such that $1 \le c \le r-1$. Then
	\begin{equation*}
		\ex_\M\left(\frac{t^r-t^{r-c}}{t-1},r,U_{2,t+2}\right) = b(\BBG{r}{t}{c}).
	\end{equation*}
\end{conjecture}

The statement holds when $t=2$ and $r\le 3$; this is immediate for $r \le 2$, while for $r=3$ it can be shown that a matroid attaining the maximum number of bases is simple, after which the statement follows from a case analysis. When $t = 2$, Conjecture~\ref{conj: bose_burton} is equivalent to the following statement.

\begin{conjecture}\label{conj: binary bose_burton}
Let $r$ and $c$ be integers so that $1 \le c \le r - 1$.
Then the $(2^r - 2^{r - c})$-element subset of $\bF_2^r$ with the most bases is the set obtained from $\bF_2^r$ by deleting a copy of $\bF_2^{r - c}$. 
\end{conjecture}

Equivalently, the set
 \[ \{ x \in \bF_2^r \setminus \{0\} : x_i = 1 \text{ for some }  i \in [c]\}\]
maximizes the number of bases over all $(2^r-2^{r-c})$-element sets.
Conjecture \ref{conj: binary bose_burton} was recently proved in the case $c = 1$ by Ellis, Ivan, and Leader \cite[Theorem 5]{EIL25}.
They in fact prove more generally that for all integers $d$ with $1 \le d \le r$, the $2^{r-1}$-element subset of $\bF_2^r$ with the most $d$-element independent sets is the set obtained from $\bF_2^r$ by deleting a copy of $\bF_2^{r-1}$, and we expect that this is also true for larger values of $c$.

\subsection{Minor-closed classes}
For a minor-closed class $\cM$ of matroids, let $\ex_{\cM}(n, r)$ be the maximum number of bases of an $n$-element, rank-$r$ matroid in $\cM$. Our proof of Theorem~\ref{thm: finite field daisy matroid Turan density} relies heavily on Theorem~\ref{thm: Kung}, which provides an upper bound on the number of elements of a rank-$r$ matroid with no $U_{2,t+2}$-minor.
The function $h_{\cM}$ that maps a positive integer $r$ to the maximum number of elements of a simple matroid in $\cM$ with rank at most $r$ is the \emph{extremal function} or \emph{growth rate function} of $\cM$.
Our proof of Theorem \ref{thm: finite field daisy matroid Turan density} leads to the following question: can $h_{\cM}(r)$ be used to find $\ex_{\cM}(n, r)$?
We give a concrete example in which this may be of interest.
Let $\cN$ be the class of binary matroids with no $\PG{2}{2}$-minor ($\PG{2}{2}$ is also commonly called the Fano plane, and denoted by $F_7$).
Kung~\cite[Corollary 6.4]{Kung1993} proved that $h_{\cN}(r) = \binom{r+1}{2}$ with equality holding for the graphic matroid of a complete graph on $r+1$ vertices.
This leads to the following natural problem.

\begin{problem}
    Let $n,r$ be integers with $n \ge r \ge 1$. Determine $\ex_{\cN}(n,r)$ for the class $\cN$ of binary matroids with no $F_7$-minor.
\end{problem}

If for every large $n$ there is an $n$-element, rank-$r$ graphic matroid $M$ with $b(M) = \ex_{\cN}(n, r)$, a solution would generalize Kelmans' result \cite{Kelmans} on the maximum number of spanning trees of an $(r+1)$-vertex graph with $n$ edges.

More generally, the Growth Rate Theorem \cite{Geelen-Kung-Whittle2009} shows that there are only four possibilities for the behavior of $h_{\cM}$ for a minor-closed class $\cM$: linear, quadratic, exponential, or infinite.
We close with the following open-ended question: does the classification of extremal functions lead to a corresponding classification of Tur\'an basis densities $\pi_{\cM}(r)$ for minor-closed classes of matroids?

\subsection*{Acknowledgements}
The authors would like to thank Dylan King, Imre Leader, and Ethan White for the helpful discussions and suggestions.  The third author was supported in part by NSF RTG DMS-1937241 and an AMS-Simons Travel Grant.

\bibliographystyle{abbrv}
\bibliography{citations}

\end{document}